\documentclass[12pt]{amsart}

\usepackage{amsmath} 
\usepackage{amsfonts}
\usepackage{amssymb}
\usepackage{amsthm}
\usepackage{latexsym}
\usepackage{enumerate}
\usepackage{mathrsfs}
\usepackage{setspace}
\usepackage{mathtools}

\allowdisplaybreaks

\newtheorem{definition}{Definition}[section]
\newtheorem{lemma}{Lemma}[section]
\newtheorem{thm}{Theorem}[section]
\newtheorem{prop}{Proposition}[section]
\newtheorem{coro}{Corollary}[section]

\newtheorem{remark}{Remark}[section]
\newtheorem*{remark*}{Remark}
\newtheorem*{thm*}{Theorem}
\numberwithin{equation}{section}

\newcommand{\pr}{\partial}
\newcommand{\veps}{\varepsilon}
\newcommand{\lm}[2]{\lim\limits_{#1\to #2}}
\newcommand{\definedas}{\mathrel{\raise.095ex\hbox{\rm :}\mkern-5.2mu=}}
 \newcommand{\asdefined}{\mathrel{=\mkern-5.2mu\raise.095ex\hbox{\rm :}}}

\def\vphi{\varphi}
\def\tr{\textmd{tr}}
\def\dv{\textnormal{div}}
\def\M{\mathscr{M}}
\def\Lap{\Delta}
\def\grad{\nabla}

\def\dint{\displaystyle\int}
\def\R{\mathbb{R}}

\def\R{\mathbb{R}}
\def\S{\Sigma}
\def\({\left(}
\def\){\right)}

\def\={\stackrel{*}{=}}
\def\s{\sigma}

\def\g{\gamma}
\def\ADM{\textnormal{ADM}}
\def\bS{\mathbb{S}}
\def\m{\mathfrak{m}}
\def\k{\kappa}

\begin{document}

\title[Asymptotically flat extensions with charge]{Asymptotically flat extensions with charge}
\author[Alaee]{Aghil Alaee}
\address{Center of Mathematical Sciences and Applications, Harvard University, 
Cambridge MA 02138, USA}
\email{aghil.alaee@cmsa.fas.harvard.edu}

\author[{Cabrera Pacheco}]{Armando J. {Cabrera Pacheco}}
\address{Department of Mathematics, Universit\"at T\"ubingen,  72076 T\"{u}bingen, Germany.}
\email{cabrera@math.uni-tuebingen.de}

\author[Cederbaum]{Carla Cederbaum}
\address{Department of Mathematics, Universit\"at T\"ubingen,  72076 T\"{u}bingen, Germany.}
\email{cederbaum@math.uni-tuebingen.de}

\begin{abstract}
The Bartnik mass is a notion of quasi-local mass which is remarkably difficult to compute. Mantoulidis and Schoen~\cite{M-S} developed a novel technique to construct asymptotically flat extensions of minimal Bartnik data in such a way that the ADM mass of these extensions is well-controlled, and thus, they were able to compute the Bartnik mass for minimal spheres satisfying a stability condition. In this work, we develop extensions and gluing tools,  \`a la Mantoulidis--Schoen, for time-symmetric initial data sets for the Einstein--Maxwell equations that allow us to compute the value of an ad-hoc notion of charged Barnik mass for suitable charged minimal Bartnik data.
\end{abstract}

\maketitle

\thispagestyle{empty}

\section{Introduction and results}

In~\cite{Bartnik-89}, motivated by the notion of electrostatic capacity of a conducting body, Bartnik proposed a new notion of quasi-local mass tailored to open sets $\Omega$ in time-symmetric, asymptotically flat, initial data sets for the Einstein equations, satisfying the dominant energy condition. This notion of quasi-local mass is known as the \emph{Bartnik mass}. We recall that initial data sets for the Einstein equations correspond to spacelike slices of spacetimes and are described by a Riemannian $3$-manifolds $(M,\g)$ together with a $(0,2)$-tensor field $K$ playing the role of the second fundamental form. The dominant energy condition implies certain properties of $(M,\g,K)$. When time-symmetry ($K=0$) is assumed, the dominant energy condition reduces to the scalar curvature $R(\g)$ being bounded below by $0$. 

We will consider the boundary version of Bartnik's mass, in which it is defined purely in terms of the boundary geometry $\S \definedas \pr \Omega$ of $\Omega$. Hence, given \emph{Bartnik data} $(\S \cong \bS^2,g,H)$, where $g$ is a Riemannian metric on $\S$ and $H \geq 0$ is a smooth function on $\S$, we consider the following set $\mathcal{A}$ of \emph{admissible extensions}: an asymptotically flat Riemannian $3$-manifold $(M,\g)$, with non-negative scalar curvature, is an admissible extension of $(\S \cong \bS^2,g,H)$ if $\pr M$ is isometric to $(\S,g)$ and has mean curvature $H$ as a submanifold of $M$; moreover, we require $\pr M$ to be outer-minimizing. Then we define the Bartnik mass of $(\S \cong \bS^2,g,H)$ as
\begin{equation*}
\m_B(\S \cong \bS^2,g,H) \definedas \inf \{ m_{\ADM}(M,\g) \, | \, (M,\g) \in \mathcal{A}  \}.
\end{equation*}
The Bartnik mass is remarkably difficult to compute. However, recall that the Riemannian Penrose inequality states that for an asymptotically flat Riemannian manifold $(M^3,\g)$ with non-negative scalar curvature and outer-minimizing minimal boundary $\pr M$, one has
\begin{equation}\label{eq-AFRPI}
m_{\ADM}(M,\g) \geq \sqrt{\frac{|\pr M|}{16 \pi}}, 
\end{equation}
where $|\pr M|$ denotes the area of the boundary $\pr M$. 

Equality holds if and only if $(M^3,\g)$ is isometric to a spatial Schwarzschild manifold. The Riemannian Penrose inequality was proven when $\pr M$ has a single connected component by Huisken and  Ilmanen~\cite{HI} using a weak formulation of the inverse mean curvature flow, motivated by an argument by Geroch in~\cite{Geroch}. Bray~\cite{Bray} proved it allowing $\pr M$ to be disconnected, using the conformal flow.

It follows readily that the \emph{Hawking mass} of $(\S\cong \bS^2,g,H)$,
\begin{equation*}
\m_H (\S\cong \bS^2,g,H)\definedas \sqrt{\frac{|\S|}{16 \pi}}\(1 - \frac{1}{16\pi} \int_{\S} H^2 \, d\s  \),
\end{equation*}
where $|\S|$ denotes the area of $\S$ and $d\s$ is the area form of $\S$ with respect to $g$, provides a lower bound for the Bartnik mass when $H \equiv 0$:
\begin{equation*}
\m_H(\pr M,g,H \equiv 0) = \sqrt{\frac{|\pr M|}{16 \pi}} \leq \m_B(\pr M,g,H \equiv 0).
\end{equation*}

Mantoulidis and Schoen~\cite{M-S} computed the Bartnik mass for \emph{minimal Bartnik data} $(\S\cong \bS^2,g,H\equiv 0)$, when $g$ satisfies that the first eigenvalue $\lambda_1$ of the operator $-\Lap_g + K(g)$, where $K(g)$ denotes the Gaussian curvature of $g$, is strictly positive. To do so, they developed a novel technique to handcraft asymptotically flat extensions of such minimal Bartnik data, in such a way that the ADM mass can be made arbitrarily close to the optimal value in~\eqref{eq-AFRPI}.

In the context of solutions to the Einstein--Maxwell equations, considering time-symmetric initial data sets amounts to study Riemannian manifolds $(M^3,\g)$ together with a vector field $E$ on $M$, acting as an electric field. The dominant energy condition then translates to requiring $R(\g) \geq 2|E|^2_{\g}$. In this setting, the charged Riemannian Penrose inequality states~\cite{Jang,D-K} that for an asymptotically flat Riemannian manifold $(M^3,\g)$ with boundary $\pr M$, assumed to be minimal and outer-minimizing, and a vector field $E$ acting as an electric field satisfying $R(\g) \geq 2 |E|^2_{\g}$, we have
\begin{equation*}
m_{\ADM}(M,\g) \geq \sqrt{\frac{|\pr M|}{16 \pi}} + \sqrt{\frac{\pi}{|\pr M|}}Q^2,
\end{equation*}
where $Q$ denotes the total charge of the time-symmetric initial data set. Equality holds if and only if $(M^3,\g)$ is isometric to a spatial Reissner--Nordstr\"om manifold (see Section~\ref{sec-intro} for the relevant definitions).

After defining an appropriate set of admissible extensions for \emph{charged Bartnik data}, that is for 4-tuples $(\S \cong \bS^2,g,H,Q)$, where $g$ is a Riemannian metric on $\S$, $H \geq 0$ is a smooth function on $\S$, and $Q \in \R$, we formulate an ad-hoc version of Bartnik mass in this context, denoted by $\m_B^{CH}$, tailored to time-symmetric initial data sets for the Einstein--Maxwell equations, satisfying the dominant energy condition. Our main result is Theorem \ref{thm-main} which can be stated --- somewhat imprecisely for now --- as follows.
\newpage
\begin{thm*} 
Let $(\S \cong \bS^2,g_o,H_o \equiv 0,Q_o)$ be minimal charged Bartnik data satisfying $\lambda_1 \definedas \lambda_1(-\Lap_{g_o}+ K(g_o))>0$, where $\lambda_1$ denotes the first eigenvalue of the operator $-\Lap_{g_{o}} + K(g_o)$, and $K(g_o)$ denotes the Gaussian curvature of $g_o$. Let $4 \pi r_o^2 \definedas |\S|$. Suppose that 
\begin{equation*}
Q_o^2 < r_o^2
\end{equation*}
and assume furthermore that
\begin{equation*} 
\kappa > \frac{Q_o^2}{r_o^4},   
\end{equation*}
where $\kappa$ is a real number depending only on $(\S,g_{o})$. Then
\begin{equation*}
\m_B^{CH}(\S \cong \bS^2,g_o,H_o \equiv 0,Q_o) = \sqrt{\frac{|\S|_{g_o}}{16 \pi}} + \sqrt{\frac{\pi}{|\S|_{g_o}}}Q_o^2.
\end{equation*}
\end{thm*}

The threshold $\kappa$ appearing in this Theorem will be given by the infimum of the first eigenvalue of the operator $-\Lap + K$ along a precise smooth path of metrics on $\S$ connecting the metric $g_o$ to a round metric, see Section~\ref{sec-minimal}. The proof of our main Theorem~\ref{thm-main} is inspired by the Mantoulidis--Schoen construction~\cite{M-S}. This construction has proven to be useful to obtain Bartnik mass estimates. Relevant and related results include those in~\cite{CM,M-X,CCMM,CCM,MWX}; for a survey on this topic see~\cite{CC-survey}.

\begin{remark*}
In this work, we construct time-symmetric initial data sets for the Einstein--Maxwell equations. To relate these initial data sets with horizon inner boundary to the Cauchy problem in general relativity, suppose that they can be suitably regularly geodesically completed by a fill-in consisting of a Riemannian ball together with a source-free electric field in such a way that the dominant energy condition is satisfied in the completion. The evolution result of Choquet-Bruhat and Friedrich~\cite{CB-F} for compact charged dust matter should then apply to the completed initial data set and lead to short time existence of a unique spacetime satisfying the Einstein--Maxwell equations for charged dust. We do not know whether the procedure to construct initial data sets presented here can be carried out in a way that gives an electro-vacuum solution, as it is in~\cite{M-S}.
\end{remark*}

This article is organized as follows. In Section~\ref{sec-intro}, we introduce basic notions for time-symmetric initial data sets for the Einstein--Maxwell equations and formulate the definition of a boundary charged Bartnik mass. In Section~\ref{sec-collars}, we study collar extensions for minimal charged Bartnik data, define the electric fields to consider along these collar extensions, and describe their interaction. Independent gluing tools for rotationally symmetric Riemannian manifolds with electric fields are obtained in Section~\ref{sec-gluing-tools}. The main theorem is then obtained in Section~\ref{sec-minimal} as a corollary from a more general result in the same spirit as Mantoulidis--Schoen's result.

\bigskip

\paragraph*{\emph{Acknowledgements.}} We extend sincere thanks to Alejandro Pe\~nuela D\'iaz for a thorough reading of the draft of this paper. All three authors thank the Erwin Schr\"{o}dinger Institute for hospitality and support during our visits in 2017 in the context of the program \emph{Geometry and Relativity} and the Banff International Research Station for hospitality and support during the workshop \emph{Asymptotically hyperbolic manifolds} in 2018. AA was supported by a post-doctoral fellowship from the Natural Sciences and Engineering Research Council of Canada (NSERC). AJCP and CC thank the Carl Zeiss foundation for its generous support. The work of CC is supported by the Institutional Strategy of the University of T\"ubingen (Deutsche Forschungsgemeinschaft, ZUK 63).

\section{Time-symmetric initial data sets for the Einstein-Maxwell equations} \label{sec-intro}

Consider a time-symmetric initial data set for Einstein--Maxwell equations, that is, a triplet $(M,\g,E)$, where $(M,\g)$ is a Riemannian manifold and $E$ is a vector field, to be interpreted as an electric field. This corresponds to considering a spacelike slice of a spacetime satisfying the Einstein--Maxwell equations with second fundamental form $K \equiv 0$. We will always assume that the charge density vanishes, that is, $\dv_{\g} E=0$, and thus charges are conserved, and that the magnetic field $B$ vanishes. By the Gauss--Codazzi equations, it then follows that $(M,\g,E)$ satisfies
\begin{align*}
R(\g)  -2|E|^2_{\g} &= 16 \pi \mu, \\
\dv_{\g} E&=0,
\end{align*}
where $\mu$ is the energy density of the non-electromagnetic matter fields. Additionally, we will always assume that the dominant energy condition $\mu \geq 0$ holds. In particular, this implies that we will be interested in Riemannian manifolds with scalar curvature bounded below by $2|E|^2_{\g}$.

\begin{definition}\label{def-elec-mfd}
We say that $(M,\g,E)$ is an \emph{electrically charged Riemannian manifold} if $(M,\g)$ is a Riemannian manifold and $E$ a smooth vector field on $M$, to  be interpreted as an electric field. We say that the electrically charged Riemannian manifold $(M,\g,E)$ \emph{satisfies the dominant energy condition} if
\begin{equation*}
R(\g) \geq 2 |E|^2_{\g}.
\end{equation*}
\end{definition}

An electrically charged Riemannian manifold $(M,\g,E)$ is \emph{asymptotically flat} if $(M,\g)$ is asymptotically flat and $E \to 0$ asymptotically and suitably fast in the asymptotically flat end. In this context we will sometimes simply refer to the vector field $E$ in Definition~\ref{def-elec-mfd} as the \emph{electric field}. 

The \emph{total charge} of an asymptotically flat, electrically charged Riemannian manifold $(M,\g,E)$ is given by
\begin{equation} \label{total-charge}
\mathbf{Q}_{(M,\g,E)} \definedas  \lm{r}{\infty} \frac{1}{4\pi} \int_{\mathbb{S}^2_r} \g(E,\nu) \, d\s_r,
\end{equation}
where $\mathbb{S}^2_{r}$ is the coordinate sphere with radius $r$ and outer unit
normal $\nu$ and $d\s_r$ denotes its area form (see, for example,~\cite{D-GH} and the references therein).

Using Stokes' theorem and our assumption that the charge density $\dv_{\g}E$ vanishes, we can define the \emph{total charge contained in $\S \subset M$}, where $\S$ is a closed  surface (homologous to the 2-sphere $\mathbb{S}^2_{\infty} \definedas \lm{r}{\infty} \mathbb{S}^2_{r}$ in the asymptotically flat end) as
\begin{equation} \label{total-charge-surface}
\mathbf{Q}_{\S} \definedas  \frac{1}{4\pi} \int_{\S} \gamma(E,\nu) \, d\sigma,
\end{equation}
where $d\s$ denotes the area form on $\S$ with respect to the induced metric. Note that by our assumptions this quantity is the same for all 2-surfaces homologous to $\S$.

Recall that the Reissner--Nordstr\"om spacetime is a static solution to the Einstein--Maxwell field equations, representing a gravitational field surrounding a static spherical black hole with charge. The \emph{spatial Reissner--Nordstr\"om manifold with charge $Q \in \R$ and mass $m \geq |Q|$} arises as the time-symmetric slice $\{ t=0 \}$ of the Reissner--Nordstr\"om spacetime, and can be described as the Riemannian manifold $M^{RN}_{m,Q}=(r_+,\infty) \times \bS^2$, with metric~$\gamma_{m,Q}$ given by
\begin{equation} \label{RN}
\g_{m,Q} \definedas \( 1- \frac{2m}{r} + \frac{Q^2}{r^2} \)^{ -1}dr^2 + r^2 g_*,
\end{equation}
where $r_+ = m + \sqrt{m^2 - Q^2}$, and $g_*$ denotes the standard round metric on $\bS^2$. The metric~$\gamma_{m,Q}$ is not defined when $r=r_+$, but as we will see later, this is just a coordinate singularity. The \emph{electric field $E_{m,Q}$} on $(M^{RN}_{m,Q},\g_{m,Q})$ given by
\begin{equation}
E_{m,Q} \definedas \frac{Q}{r^2}\sqrt{1- \frac{2m}{r} + \frac{Q^2}{r^2}}\, \pr_r
\end{equation}
satisfies the source-free condition $\dv_{\g_{m,Q}} E_{m,Q}=0$ and the time-symmetric, electro-vacuum constraint equation $R(\g_{m,Q})=2|E_{m,Q}|^2_{\g_{m,Q}}$. In particular, $(M^{RN}_{m,Q},\g_{m,Q},E_{m,Q})$ is an asymptotically flat, electrically charged Riemannian manifold satisfying the dominant energy condition. If $m=|Q|$, the spacetime is called an \emph{extremal} Reissner--Nordstr\"om black hole and its initial data set $(M,\gamma_{m,Q},E_{m,Q})$ has a cylindrical end geometry, which is often thought of as an infinite `throat' region. This extreme geometry appears as the rigidity case in the area-charge inequality and the positive mass theorem for charge black holes (for a comprehensive review, see~\cite{D-GH}). In this work, we will only be interested in the case when $m > |Q|$ which is called the \emph{sub-extremal} case.

By performing the change of variables
\begin{equation*}
s(r) \definedas \int_{r_+}^{r}  \( 1- \frac{2m}{t} + \frac{Q^2}{t^2} \)^{ -1/2} \, dt,
\end{equation*}
we can extend the (sub-extremal) metric~$\gamma_{m,Q}$ given in~\eqref{RN} to include the \emph{horizon boundary} $\{ s=0 \}$, and write it as
\begin{equation*}
\g_{m,Q} = ds^2 + u_{m,Q}(s)^2 g_*
\end{equation*}
on $[0,\infty) \times \bS^2$.
\newpage
\noindent The radial profile function $u_{m,Q}\colon[0,\infty) \to [r_+,\infty)$ satisfies
\begin{enumerate}
\item $u_{m,Q}(0)=r_+$,
\item $u_{m,Q}'(s)= \( 1- \frac{2m}{u_{m,Q}(s)} + \frac{Q^2}{u_{m,Q}(s)^2} \)^{1/2}$, and
\item $u_{m,Q}''(s)=\frac{m u_{m,Q}(s) - Q}{u_{m,Q}(s)^3}$.
\end{enumerate}
The electric field $E_{m,Q}$ in these coordinates is given by 
\begin{equation*}
E_{m,Q}=\frac{Q}{u_{m,Q}^{2}} \pr_s,
\end{equation*}
and we have 
\begin{equation} \label{scalar-RN}
R(\g_{m,Q})=2|E_{m,Q}|^2_{\g_{m,Q}} = \frac{2 Q^2}{u_{m,Q}^4}\,.
\end{equation}

The well-known Penrose inequality relates the total mass of a spacetime with the area of the black holes contained in it, its general form is still an open problem. At the level of time-symmetric initial data sets for the Einstein--Maxwell equations, the following version was first discussed by and proved by Jang in~\cite{Jang}, assuming smoothness of the solution of the inverse mean curvature flow equation (IMCF). This assumption is now superfluous if one exchanges the smooth IMCF equation for Huisken--Ilmanen's weak IMCF formulation~\cite{HI}, see~\cite{Mars}. The corresponding rigidity statement was proved by Disconzi and Khuri in~\cite{D-K}. A closed $2$-surface in an asymptotically flat Riemannian manifold $(M,\g)$ is called a \emph{horizon} if its mean curvature vanishes. We say that it is \emph{outer-minimizing} if it minimizes area among all surfaces enclosing it and homologous to it. In the terminology introduced above, these results can be summarized as follows.

\begin{thm}[Riemannian Penrose inequality with charge] \label{thm-RPI}
Let $(M,\g,E)$ be an asymptotically flat, electrically charged Riemannian $3$-manifold  with a connected outer-minimizing horizon boundary $\pr M$. Assume further that the charge density is zero, i.e., $\dv_{\g} E=0$, and that the dominant energy condition $R(\g) \geq 2 |E|^2_{\g}$ is satisfied. Then, 
\begin{equation}\label{RPI}
m_{\ADM}(M,\g) \geq \sqrt{\frac{A}{16 \pi}} + \sqrt{\frac{\pi}{A}}Q^2,
\end{equation}
where $Q\definedas \mathbf{Q}_{(M,\g,E)}$. Equality holds if and only if $(M,\g,E)$ is isometric to a sub-extremal Reissner-Nordstr\"om manifold.
\end{thm}

Given an electrically charged Riemannian $3$-manifold $(M,\g,E)$ and a closed $2$-surface $\S \subset M$, following~\cite{D-K} we define the \emph{charged Hawking mass} of $\S$, $\m_{H}^{CH}(\S)$ as
\begin{equation} \label{eq-c-hawking}
\m_{H}^{CH}(\S) \definedas \sqrt{\frac{|\S|}{16 \pi}}\( 1 + \frac{4 \pi Q^2}{|\S|} - \frac{1}{16 \pi}\int_{\S} H^2 \, d\s \),
\end{equation}
where $|\S|$ denotes the area of $\S$ with respect to the metric induced on $\S$ by $\g$, $Q$ is the charge contained in $\S$, $H$ is the mean curvature of $\S$ and $\s$ is the area form on $\S$ with respect to the metric induced by $\g$. Notice that when $H=0$ (i.e., at a horizon), we recover the right hand side of the Penrose inequality~\eqref{RPI}. 

\subsection{Charged Bartnik mass}

We now proceed to formulate an ad-hoc definition of \emph{charged Bartnik mass}. We start by defining what we will be referring to as \emph{charged Bartnik data}.

\begin{definition} 
A 4-tuple $(\S\cong \bS^2,g,H,Q)$ is called \emph{charged Bartnik data}, if $g$ is a Riemannian metric on $\S$, $H \geq 0$ is a smooth function on $\S$, and $Q$ is a real number. 
\end{definition}

Notice that for such data, it makes sense to consider its charged Hawking mass as a number depending only on the given data via~\eqref{eq-c-hawking}, as
\begin{equation}
\m_{H}^{CH}(\S,g,H,Q) \definedas \sqrt{\frac{|\S|_g}{16 \pi}}\( 1 + \frac{4 \pi Q^2}{|\S|_g } -  \frac{1}{16 \pi}\int_{\S} H^2 \, d\s \).
\end{equation}
Mimicking~\cite{Bartnik-89}, we call a triplet $(M,\gamma,E)$ an \emph{admissible extension} of the charged Bartnik data $(\S\cong \bS^2,g,H,Q)$, if $(M,\gamma)$ is an asymptotically flat Riemannian manifold with outer-minimizing boundary $\pr M$ isometric to $(\S,g)$ and with mean curvature $H$, $E$ is a smooth vector field on $M$, interpreted  as an electric field, such that
\begin{equation}
R(\g)\geq 2|E|^2_{\g},\qquad \text{div}_{\g}E=0,
\end{equation}
and $Q$ is the total charge of $\pr M$ defined by~\eqref{total-charge-surface}. Denoting the set of admissible extensions of $(\S \cong \bS^2,g,H,Q)$ by $\mathcal{A}_{g,H,Q}$, the charged Bartnik mass is naturally defined as
\begin{equation}
\mathfrak{m}_{B}^{CH}(\S,g,H,Q) \definedas \inf \{m_{\ADM}(M,\gamma) \, | \,  (M,\gamma,E) \in \mathcal{A}_{g,H,Q} \}.
\end{equation}

Recall that in the usual setting, one could require instead of the outer-minimizing boundary condition that there are no minimal surfaces in the extensions (homologous to the boundary), except possibly the boundary. For this work, it makes no difference which condition is chosen since the extensions considered here will satisfy both.

It follows readily from Theorem~\ref{thm-RPI} that in the case of \emph{minimal charged Bartnik data}  $(\S \cong \bS^2,g,H\equiv 0,Q)$, we have
\begin{equation*}
\m_H^{CH}(\S,g,H \equiv 0, Q) \leq \m_B^{CH}(\S,g,H \equiv 0, Q).
\end{equation*}
\vspace{1ex}

Let us remark that the definition of charged Bartnik mass used here is an ad hoc analogy to the boundary version of the definition of Bartnik mass in the uncharged context. To formulate in detail a definition for the Bartnik mass for a connected domain $\Omega$ in an initial data set for the Einstein--Maxwell equations as in~\cite{Bartnik-89}, it would be necessary to consider extensions satisfying a weak dominant energy condition (in the sense of low regularity). The required analysis to do so would lead too far from the scope and goals of this article.
\newpage
\section{Collar extensions with an electric field} \label{sec-collars}
Given charged Bartnik data $(\S \cong \bS^2,g,H,Q)$, we will be interested in constructing appropriate ``collar extensions''. To be precise, a \emph{collar extension} will be a metric on $[0,1] \times \S$ of the form
\begin{equation} \label{collar}
\g = v(t,\cdot)^2 dt^2 + F(t)^2 g(t)
\end{equation}
together with a suitable electric field $E\parallel \partial_{t}$, where $v$ is a positive smooth function on $[0,1] \times \bS^2$, $F$ is a positive smooth function on $[0,1]$, and $\{ g(t) \}_{0 \leq t \leq 1 }$ is a suitable smooth path of metrics connecting $g$ to a round metric on $\bS^2$. The functions and the path of metrics are chosen so that the level $t=0$ is isometric to the given charged Bartnik data.

In this work, we consider the case when the data $(\S \cong \bS^2,g,H \equiv 0,Q)$ is prescribed on the boundary of the collar extension (as a minimal surface). The collar extensions are constructed inspired in the work of Mantoulidis and Schoen~\cite{M-S}. 

The smooth path of metrics $\{ g(t) \}_{0 \leq t \leq 1}$ is chosen so that it preserves a curvature condition of the given data (see Lemmas~\ref{path-lambda}) and, in addition, it is required to satisfy the following conditions:
\begin{enumerate}[(i)]
\itemsep0.5em
\item $g(0)=g$ and $g(1)$ is round,
\item $g'(t)=0$ for $t \in [\theta,1]$ where $0 < \theta < 1$, and
\item $\tr_{g(t)} g'(t)=0$ for $t \in [0,1]$, i.e., the area-form is preserved.
\end{enumerate}
Notice that by (iii), in particular the area is also preserved, so $|\S|_{g(t)} = 4\pi  r_o^2$, for some $r_o > 0$, which is called the area-radius. As a consequence $g(1)=r_o^2 g_*$, where $g_*$ denotes the standard metric on $\bS^2$.

The existence of such a path follows from the uniformization theorem together the procedure in~\cite[Lemma 1.2]{M-S} to ensure that all conditions above are satisfied.

Recall that the scalar curvature of the collar metric~\eqref{collar} is given (see for example~\cite{CCMM,CC-survey}) by
\begin{equation}\label{scalar-collar}
\begin{split} 
R(\g) &= 2v(t,\cdot)^{-1}\left[ -\Lap_{F(t)^2 g(t)} v(t,\cdot) + \frac12 R(F(t)^2 g(t)) v(t,\cdot)  \right] \\
&\qquad + v(t,\cdot)^{-2}\left[ \dfrac{-2F'(t)^2 - 4F(t)F(t)''}{F(t)^2}  - \frac14 |g'(t)|^2_{g(t)}  + 4 \frac{\pr_t v(t,\cdot)}{v(t,\cdot)} \frac{F'(t)}{F(t)}\right],
\end{split}
\end{equation}
and the mean curvature of a $t$-constant slice $\S_t=\{ t \} \times \S$ is given by
\begin{equation} \label{mean-collar}
H(t,\cdot)= \frac{2 F'(t)}{v(t,\cdot) F(t)}.
\end{equation}

We will be interested in constructing collar extensions together with a divergence free vector field $E$, playing the role of an electric field. The next lemma gives a way to constructing them exploiting the properties of the path $\{ g(t) \}$.

\begin{lemma} \label{lemma-e-field-collar}
Let $(\S \cong \bS^2,g)$ be Riemannian 2-manifold, $v\colon[a,b] \times \bS^2 \to \R$ be a positive smooth function and $F\colon[a,b] \to \R$ be a smooth positive function. Let $\g$ be the metric on $[a,b] \times \bS^2$ given by
\begin{equation} \label{collar-indvlemma}
\g \definedas v(t,\cdot)^2 dt^2 + F(t)^2 g(t),
\end{equation}
where $\{ g(t) \}$ is a smooth path of metrics as in satisfying (i)-(iii) above. Then the vector field defined as
\begin{equation} \label{ef-indvlemma}
E \definedas \dfrac{Q}{r_o^2 v(t,\cdot) F(t)^2} \pr_t,
\end{equation}
where $Q$ is any constant, satisfies
\begin{equation} \label{dvEzero}
\dv_{\g} E = 0.
\end{equation}
In addition, for any $t \in [a,b]$ we have
\begin{equation} \label{ec-indvlemma}
\mathbf{Q}_{\S_t} = \frac{1}{4 \pi}\dint_{\S_t} \gamma(E,\nu) \, dA_{F(t)^2g(t)} = Q,
\end{equation}
where $\S_t = \{ t \} \times \S$ and $\nu$ is the unit normal vector to $\S_t$ pointing in the direction of $\pr_t$.
\end{lemma}

\begin{proof}
Fix a system of coordinates on $\bS^2$. In what follows, we use Latin indices for the coordinates in $[a,b]\times \bS^2$. From the definition of $E$ we have
\begin{align*}
\dv_{\g} E &= \dfrac{1}{\sqrt{|\g|}}\pr_{i}(E^{i} \sqrt{|\g|} ) \\
&=\pr_t E^{t} + \dfrac{E^t}{2}\textnormal{trace}( \g^{-1} \pr_t \g),
\end{align*}
where $\textnormal{trace}( \g^{-1} \pr_t \g)$ denotes the trace of the matrix $( \g^{-1} \pr_t \g)$. Recalling that $\tr_{g(t)}g'(t) =0$, we have
\begin{align*}
\dfrac{E^t}{2}\tr( \g^{-1} \pr_t \g) &= \dfrac{Q}{r_o^2}\(\frac{v'(t, \cdot)}{v(t,\cdot)^2 F(t)^2} + 2 \frac{F'(t)}{v(t, \cdot)F(t)^3} \) = -\pr_t E^t,
\end{align*}
and~\eqref{dvEzero} follows. To check~\eqref{ec-indvlemma}, note that $\nu = v^{-1} \pr_t$ and thus
\begin{align*}
\mathbf{Q}_{\S_t} = \frac{1}{4 \pi} \int_{\S_t} \gamma(E,\nu) \, dV_{F(t)^2g(t)} &=\dfrac{Q}{4 \pi  r_o^2}  \int_{\S_t} dA_{g(t)} = Q.
\end{align*}
\end{proof}

As pointed out by McCormick, Miao, and the second and third named authors in~\cite{CCMM}, for time-symmetric initial data sets for the (uncharged) Einstein equations, in order to glue a collar extension of given data to a Schwarzchild manifold, it is useful to look at the growth of the Hawking mass along the collar extension. However, in the minimal case this reduces to controlling the area growth along the collar as in~\cite{M-S}; since we are dealing with minimal charged Bartnik data, the collar extensions considered in~\cite{M-S} will suffice for our goals (see Section~\ref{sec-minimal}). We remark that one could consider \emph{charged CMC Bartnik data}, i.e., when $H$ is a positive constant, and use the type of collar extensions constructed by Miao, Wang, and Xie in~\cite{MWX} (which improve those in~\cite{M-X}), which are expected to provide good control of the charged Hawking mass along the collar extension, to obtain the corresponding estimates together with the gluing tools developed in Sections~\ref{sec-gluing-tools}. We do not pursue this idea here.

\section{Gluing methods for rotationally symmetric electrically charged manifolds} \label{sec-gluing-tools}

In this section, we provide tools to glue two rotationally symmetric manifolds with corresponding electric fields satisfying the dominant energy condition, in such a way that the resulting manifold is an electrically charged Riemannian manifold satisfying the dominant energy condition. The construction is inspired on the Mantoulidis--Schoen construction and thus we follow the general procedure in~\cite{M-S}. Since some of the results may have independent interest, we prove them in general dimensions.

The following lemma characterizes a class of admissible electric fields for a rotationally symmetric Riemannian manifold, in the context of initial data sets with charge. Even though the choice of the electric field in the following result is somewhat arbitrary, it is reasonable to impose this form since it resembles the symmetries of the electric field of an spatial Reissner--Nordstr\"om manifold.

\begin{lemma} \label{lemma-DEC-rotsym}
Let $f\colon[a,b] \to \R$ be a smooth positive function. Define the  Riemannian metric $\g_f \definedas ds^2 + f(s)^2 g_*$ on the manifold $[a,b] \times \bS^n$, where $g_*$ denotes the standard metric on $\bS^n$. Then, the following holds. 
\begin{enumerate}[(a)] 
\item The vector field 
\begin{equation*}
E \definedas \frac{Q}{f(s)^n} \pr_s,
\end{equation*}
where $Q$ is any constant, satisfies
\begin{equation*}
\dv_{\g} E = 0.
\end{equation*}

\item $R(g) \geq n(n-1)|E|^2_{\g}$ if and only if 
\begin{equation}
f'' \leq \dfrac{n-1}{2f}\( 1 - (f')^2 - \dfrac{Q^2}{ f^{2(n-1)}}   \) \label{RgeqE2}
\end{equation}
\end{enumerate}
\end{lemma}

\begin{proof}
Note that the proof of (a) follows in the same way as the proof of~\eqref{dvEzero} in Lemma~\ref{lemma-e-field-collar} (with $v \equiv 1$).

For (b), first note that
\begin{equation*}
|E|^2_{\g} = \dfrac{Q^2}{f(s)^{2n}}.
\end{equation*}
The scalar curvature of $\g$ (see e.g.,~\cite[Equation (4.13)]{CM}) is given by
\begin{equation} \label{scalar-rot}
R(\g) = \dfrac{n}{f^2}[(n-1)-(n-1) (f')^2 - 2 f f''].
\end{equation}
Imposing $R(\g) \geq  \frac{n(n-1)Q^2}{f(s)^{2n}}$ leads to the inequality~\eqref{RgeqE2}.
\end{proof}

\begin{remark}
When $n=2$, part (b) in Lemma~\ref{lemma-DEC-rotsym} means that the electrically charged manifold $([a,b] \times \bS^n,\g_f)$ satisfies the dominant energy condition. Note that when $f=u_{m,Q}$ we have equalities in (b).
\end{remark}

The next lemma provides a way to smoothly glue two rotationally symmetric manifolds together with their respective electric fields. By adding stronger hypotheses in~\cite[Lemma 2.1]{CCMM} (cf.~\cite[Lemma 2.2]{M-S}), we are able to construct a rotationally symmetric bridge between the two manifolds together with an electric field.

\begin{lemma} \label{lemma-gluing}
Let $f_i \colon[a_i,b_i] \to \R$ (for $i=1,2$) be two smooth positive functions, $Q_i$ ($i=1,2$) be two constants and let $n \geq 2$. Suppose that
\begin{enumerate}
\itemsep0.25em
\item The metrics $\g_i \definedas ds^2 + f_i(s)^2 g_*$ on $[a_i,b_i] \times \bS^n$ satisfy
\begin{equation*}
R(\g_i) > n(n-1)|E_i|^2_{\g_i},
\end{equation*}
where $E_i \definedas \dfrac{Q_i}{f(s)^n} \pr_s$,
\item $f_1(b_1) < f_2(a_2)$ and $f_2'(a_2) \leq f_1'(b_1)$,
\item $f_1(b_1)^{n-1} > |Q_1|$, and $1+\frac{Q_1^2}{f_1(b_1)^{2(n-1)}} - \frac{2|Q_1|}{f_1(b_1)} > f_1'(b_1)^2$, and
\item $f_2(a_2)^{n-1} > |Q_2|$, and $1+\frac{Q_2^2}{f_2(a_2)^{2(n-1)}} - \frac{2|Q_2|}{f_2(a_2)} > f_2'(a_2)^2$.

\end{enumerate}
Then, after an appropriate translation of the interval $[a_2,b_2]$ so that 
\begin{equation}
\begin{cases}
(a_2-b_1)f_1'(b_1)=f_2(a_2)-f_1(b_1), \text{ if $f_1'(b_1)=f_2'(a_2)$,} \\
(a_2-b_1)f_1'(b_1) > f_2(a_2) - f_1(b_1) > (a_2-b_1)f_2'(a_2), \text{ if $f_1'(b_1) \geq f_2'(a_2)$,} 
\end{cases}
\end{equation}
there exists a smooth positive function $f\colon[a_1,b_2] \to \R$ such that 
\begin{enumerate}[(i)]
\itemsep0.5em
\item $f \equiv f_1$ on $[a_1,\frac{a_1+b_1}{2}]$,
\item $f \equiv f_2$ on $[\frac{a_2+b_2}{2},b_2]$,
\item the scalar curvature of $\g \definedas  ds^2 + f(s)^2 g_*$ satisfies 
\begin{equation*}
R(\g) >n(n-1)|E|^2_{\g},
\end{equation*}
where $E \definedas \dfrac{Q}{f(s)^n}\pr_s$, where $Q$ is any number such that $Q^2 \leq \min\{ {Q_1^2,Q_2^2} \}$.
\end{enumerate}
\end{lemma}

\begin{proof}
Let $\zeta\colon[b_1,a_2] \to \R$ be a smooth function such that
\begin{enumerate}
\itemsep0.5em
\item $\zeta(b_1) = f_1'(b_1)$,
\item $\zeta(a_2) = f_2'(a_2)$,
\item $\zeta' \leq 0$, and
\item $\int_{b_1}^{a_2} \zeta(x) \, dx = f_2(a_2) - f_1(b_1)$.
\end{enumerate}
Define the function $\widehat f \colon [b_1,a_2] \to \R$ by
\begin{equation*}
\widehat f(s) \definedas f_1(b_1) + \dint_{b_1}^{s} \zeta(x) \, dx.
\end{equation*}
Then the function 
\begin{equation*}
\widetilde f (s) \asdefined
\begin{cases}
f_1 \textnormal{ on $[a_1,b_1]$} \\
\hat f \textnormal{ on $[b_1,a_2]$} \\
f_2 \textnormal{on $[a_2,b_2]$}
\end{cases}
\end{equation*}
is $C^{1,1}$ on $[a_1,b_2]$.

Our goal is to consider an appropriate mollification $f_{\veps}$ of $\widetilde f$ so that
\begin{equation*}
\g_{\veps} \definedas ds^2 + f_{\veps}(s)^2 g_*,
\end{equation*}
satisfies $R(\g_{\veps}) > n(n-1)|E_{\veps}|^2_{\g_{\veps}}$, for the vector field $E_{\veps}$ defined as $E_{\veps} \definedas Qf_{\veps}^{-n} \pr_s$, where the charge $Q$ satisfies $Q^2 \leq \min\{ {Q_1^2,Q_2^2} \}$. By Lemma~\ref{lemma-DEC-rotsym}, we know that for a smooth metric $\g = ds^2 + f(s)^2 g_*$ and a vector field $E=Q f^{-n} \pr_s$, we have $R(\g) > n(n-1) |E|^2_{\g}$ if and only~if
\begin{equation*}
f'' <  \dfrac{n-1}{2 f}\(1 - (f')^2 - \dfrac{Q^2}{f^{2(n-1)}}  \).
\end{equation*}
We now define 
\begin{equation*}
\Omega[f] \definedas\dfrac{n-1}{2 f}\( 1 - (f')^2 - \dfrac{Q^2}{f^{2(n-1)}}  \).
\end{equation*}
Note that since $\widetilde f = f_i$ on $[a_i,b_i]$ ($i=1,2$),  we have
\begin{align*}
\Omega[\widetilde f] 
&=  \dfrac{n-1}{2 f_i}\( 1 - (f_i')^2 - \dfrac{Q^2}{f_i^{2(n-1)}}  \) \\
&=\Omega[f_i] \\
&>f_i'' \\
&= \widetilde f''.
\end{align*}
on $[a_1,b_1) \cup (a_2,b_2]$. Notice that conditions (3) and (4) imply that $\Omega[\widetilde{f}](b_1) > 0$ and $\Omega[\widetilde{f}](a_2)>0$, respectively. Indeed, using the fact that $Q^2<Q_1^2$ we have
\begin{align*}
\frac{2f_1(b_1)}{n-1}\Omega[f_1](b_1)&=1- f_1'(b_1)^2 - \frac{Q^2}{ f_1(b_1)^{2(n-1)}}\\
	&> 1 - \( 1+\frac{Q_1^2}{f_1(b_1)^{2(n-1)}} - \frac{2|Q_1|}{f_1(b_1)}   \) - \frac{Q^2}{ f_1(b_1)^{2(n-1)}} \\
	&=  -\frac{Q_1^2}{f_1(b_1)^{2(n-1)}} + \frac{2|Q_1|}{f_1(b_1)}  - \frac{Q_1^2}{ f_1(b_1)^{2(n-1)}} \\
	&=\frac{2|Q_1|}{f_1(b_1)^{2(n-1)}}\( f_1(b_1)^{(n-1)}  - |Q_1|\)>0.
\end{align*}
Similarly, one can check that $\Omega[\widetilde{f}](a_2)>0$.

On $(b_1,a_2)$, $\widetilde f'' \leq 0$, then $\widetilde f'$ is non-increasing. Moreover, the infimum of $\widetilde f$ on  $[b_1,a_2]$ is $f_1(b_1)$ and its maximum is attained at some $s_*$ in $(b_1,a_2]$. It follows that $f_1'(b_1) \geq \widetilde f'(s) \geq 0$ on $(b_1,s_*]$, using the positivity of $\Omega[\widetilde{f}](b_1)$, we have
\begin{equation*}
\dfrac{Q^2}{ \widetilde f(t)^{2(n-1)}} \leq \dfrac{Q_1^2}{ f_1(b_1)^{2(n-1)}}< 1 - f_1'(b_1)^2 \leq 1 - \widetilde f'(s)^2,
\end{equation*}
for $b_1 <s \leq s_*$. On the other hand,if $s_* < a_2$, on $[s_*,a_2)$, we have $0 \geq \widetilde f'(s) \geq f_2'(a_2)$, and hence using the positivity of $\Omega[\widetilde{f}](a_2)$, we have
\begin{equation*}
\dfrac{Q^2}{\widetilde f(t)^{2(n-1)}}\leq\dfrac{Q_2^2}{f_2(a_2)^{2(n-1)}} < 1 - f_2'(a_2)^2 \leq 1 - \widetilde f'(s)^2,
\end{equation*}
for $s_* \leq s < a_2$. This shows that $\Omega[\widetilde f] > 0$ on $(b_1,a_2)$. Therefore, $\Omega[\widetilde f] > \widetilde f''$ on $[a_1,b_2] \setminus \{ b_1, a_2 \}$. 

Let $d>0$ be given by
\begin{equation*}
3d \definedas \inf\limits_{[a_1,b_2] \setminus \{ b_1, a_2 \} } \( \Omega[\widetilde f] - \widetilde f''  \),
\end{equation*}
so that 
\begin{equation*}
\widetilde f'' + 3d\leq \Omega[\widetilde f],
\end{equation*}
where $\widetilde f''$ is defined.
Consider now a mollification of $\widetilde f$, $f_{\veps}$, fixing a neighborhood of the boundaries, say
\begin{equation*}
\textnormal{ $f_{\veps} \equiv f_1$  on $\left[a_1,\dfrac{a_1+b_1}{2} \right]$, and $f_{\veps} \equiv f_2$ on $\left[ \dfrac{a_2+b_2}{2}, b_2 \right]$.} 
\end{equation*}
One can explicitly define this mollification as in~\cite[Equation (4.19)]{CM}. In addition, one can check that $f_{\veps} \to \widetilde f$ in $C^1([a_2,b_2])$ as $\veps \to 0^+$, which in turns implies that $\Omega[f_{\veps}] \to \Omega[\widetilde f]$ in $C^0([a_2,b_2])$ as $\veps \to 0^+$. Then, for $\veps$ sufficiently small, we have
\begin{equation*}
\sup\limits_{[a_1,b_2]} \left| \Omega[\widetilde f] - \Omega[f_{\veps}]  \right| < d,
\end{equation*}
and consequently,
\begin{equation*}
f_{\veps}''(s) < \Omega[f_{\veps}](s) = \dfrac{n-1}{2 f_{\veps}(s)}\( 1 - f_{\veps}'(s)^2 - \dfrac{Q^2}{ f_{\veps}(s)^{2(n-1)}}  \),
\end{equation*}
as desired. Note that the electric field is then given by $E_{\veps}\definedas\dfrac{Q}{f_{\veps}(s)^n}\pr_s$.
\end{proof}

\begin{remark}
Note that by Remark~\ref{lemma-DEC-rotsym}, the vector fields considered in the hypothesis and conclusion of Lemma~\ref{lemma-gluing} are divergence-free.
\end{remark}

\begin{remark}
If in addition we require $f_i' > 0$ on $[a_i,b_i]$ for $i=1,2$ in Lemma~\ref{lemma-gluing}, the resulting bridging function $f$ satisfies $f' > 0$, which implies that the resulting manifold $[a_1,b_2] \times \bS^n$ with metric $\g = dt^2 + f(s)^2 g_*$ is foliated by positive constant mean curvature spheres.
\end{remark}

\begin{remark} \label{mCH-cond}
We can relate conditions (2) and (3) of Lemma~\ref{lemma-gluing} to a Reissner--Nordstr\"om manifold as follows. Suppose that $n=2$, then conditions (2) and (3) for $f=u_{m,Q}$ translate to
\begin{equation*}
u_{m,Q}(s) > |Q| \quad \text{ and } \quad 1+\frac{Q^2}{u_{m,Q}(s)^{2}} - \frac{2|Q|}{u_{m,Q}(s)}>u_{m,Q}'(s)^2.
\end{equation*}
The first statement is always true in a sub-extremal Reissner--Nordstr\"om manifold of mass and charge $m>|Q|$, since $u_{m,Q}(0)> |Q|$. Moreover, by definition of charged Hawking mass for the sub-extremal Reissner--Nordstr\"om we have $\m_H^{CH}(\S_s)=m>|Q|$ which is the second statement.
\end{remark}

Recall that we are interested in smoothly gluing two electrically charged Riemannian manifolds satisfying the dominant energy condition. Lemma~\ref{lemma-gluing} gives a way to smoothly glue two rotationally symmetric manifolds with electric fields satisfying certain conditions. The following lemma gives a way to modify a manifold with $R(\g) \geq n(n-1)|E|_{\g}^2$ to achieve strict inequality in a small region (cf.~\cite[Lemma 3.2]{CCM}). It will be applied to a $3$-dimensional Reissner--Nordstr\"om manifold, where $R(\g) = 2|E|_{\g}^2$ holds. 

\begin{lemma} [Bending lemma]\label{bending}
	Let $\gamma=ds^2+f(s)^2g_{\ast}$ be a metric on a cylinder $[a,\infty)\times \bS^n$, where $f(s)$ is a smooth positive function and $g_{\ast}$ is the standard round metric on $\bS^n$. Assume that the scalar curvature of $\gamma$ satisfies $R(\gamma)\geq  n(n-1)|E|^2_{\gamma}$ for the vector field $E=\frac{Q}{f(s)^{n}}\partial_s$, where $Q$ is a constant. Then for any $s_0>a$ with $f'(s_0)>0$, there is a $\delta>0$ and a metric $\widetilde{\gamma}=ds^2+\widetilde{f}(s)^2g_{\ast}$, such that $\widetilde{\gamma}=\gamma$ on $[s_0,\infty)\times \bS^n$ and $R(\widetilde{\gamma})>n(n-1)|\widetilde{E}|^2_{\gamma}$ on $[s_0-\delta,s_0]\times \bS^n$, where $\widetilde{E}=\frac{Q}{\widetilde{f}(s)^{n}}\partial_s$. If in addition, $f(s_0)>\alpha$ for some positive constant $\alpha$ and $f''(s_0)>0$, then $\widetilde{f}(s_0-\delta)>\alpha$ and $\widetilde{f}'(s-\delta)<\widetilde{f}'(s_0)$.
\end{lemma}
\begin{proof}
The proof uses the main idea for the deformation used in~\cite[Lemma 2.3]{M-S}. Note that the result holds trivially if at $s=s_0$ we have $R(\gamma)> n(n-1)|E|^2_{\gamma}$. Otherwise, we need to deform the metric to increase scalar curvature. Consider the function $\sigma$ on $s\in[s_0-\delta,s_0)$, for some $\delta \leq s_0 - a$ to be determined, defined as
\begin{equation}
\sigma(s)=\int_{s_0-\delta}^{s}\left(1+e^{-(t-s_0)^{-2}}\right)dt+K_{\delta},
\end{equation}
where $K_{\delta}$ is a positive constant so that $\sigma(s_0)=s_0$, and let $\sigma(s)=s$ for $s\geq s_0$. Define the metric $\widetilde{\gamma}=ds^2+f(\sigma(s))^2g_{\ast}\asdefined ds^2+\widetilde{f}(s)^2g_{\ast}$. Using again~\eqref{scalar-rot} together with $R(\gamma)\geq n(n-1)|E|_{\gamma}^2$, we obtain
\begin{align*}
	R(\widetilde{\gamma})&-n(n-1)|\widetilde{E}|_{\widetilde{\gamma}}^2\\
	&=\frac{n}{\widetilde{f}(s)^2}\( (n-1) - (n-1)\left[\widetilde{f}'(s)\right]^2 -2\widetilde{f}(s)\widetilde{f}''(s) \)  -\frac{n(n-1)Q^2}{\widetilde{f}(s)^{2n}}\ \\
	&=\frac{n}{f(\sigma(s))^2}\left((n-1)-(n-1)\dot{\sigma}(s)^2-2f(\sigma(s))f'(\sigma(s))\ddot{\sigma}(s)\right)\\
	&\qquad +\frac{n\dot{\sigma}^2}{f(\sigma(s))^2}\left((n-1)-(n-1)\left[f'(\sigma(s))\right]^2-2f(\sigma(s))f''(\sigma(s))\right)-\frac{n(n-1)Q^2}{f(\sigma(s))^{2n}}\\
	&\geq \frac{n}{f(\sigma(s))^2}\left((n-1)-(n-1)\dot{\sigma}(s)^2-2f(\sigma(s))f'(\sigma(s))\ddot{\sigma}(s)\right)+(\dot{\sigma}^2-1)\frac{n(n-1)Q^2}{f(\sigma(s))^{2n}},\\
\end{align*}
where $s$-derivatives and $\sigma$-derivatives are denoted by $\dot{}$ and $'$, respectively. For $s\in[s_0-\delta,s_0)$, we then have
\begin{align*}
	R(\widetilde{\gamma})-n(n-1)|\widetilde{E}|_{\widetilde{\gamma}}^2&\geq e^{-(s-s_0)^{-2}}\left(\frac{n(n-1)Q^2}{f(\sigma(s))^{2n}}-\frac{n(n-1)}{f(\sigma(s))^2}\right)\left(
	e^{-(s-s_0)^{-2}}+2\right)\\
	&\qquad+\frac{e^{-(s-s_0)^{-2}}4nf(\sigma(s))f'(\sigma(s))}{(s_0-s)^{3}}.
\end{align*} 
If we select $\delta$ sufficiently small, the second term in right-hand side dominates as $\delta\to 0$ and we obtain
	\begin{equation}
	R(\widetilde{\gamma})-n(n-1)|\widetilde{E}|_{\widetilde{\gamma}}^2>0,
	\end{equation}
on $[s_0-\delta,s_0)$. This complete the first part of the claim.  Moreover, if in addition $f(s_0) > \alpha > 0$,  choosing $\delta$ sufficiently small gives $\widetilde{f}(s_0-\delta)>\alpha$, by continuity. Finally, observe that
	\begin{equation}
	\frac{d^2}{ds^2}\widetilde{f}(s)=f''(\sigma(s))\dot{\sigma}(s)^2+f'(\sigma(s))\ddot{\sigma}(s)
	\end{equation}on $[s_0 -\delta, s_0]$. Then, if $\frac{d^2}{ds^2}{f}(s_0)>0$, by selecting $\delta$ sufficiently small, we have  $\frac{d^2}{ds^2}\widetilde{f}(s)>0$. As a result, $\frac{d}{ds}\widetilde{f}(s)$ is increasing on $[s_0-\delta, s_0]$, which implies $\frac{d}{ds}\widetilde{f}(s_0-\delta)<\frac{d}{ds}\widetilde{f}(s_0)=f'(s_0)$.
\end{proof}

\subsection{Smooth gluing to a Reissner--Nordstr\"om space}

The next proposition will allow us to glue a rotationally symmetric manifold with an electric field to an exterior region of a Reissner--Nordstr\"om manifold, in such a way that the ADM mass of the resulting manifold is controlled. For simplicity and since we will apply this result to obtain $3$-dimensional electrically charged Riemannian manifolds, we set $n=2$.

\begin{prop}\label{glue-collar-RN}
	Let $\g_f = ds^2 + f(s)^2 g_*$ be a metric on $[a,b] \times \bS^2$ and $\S_s = \{s \} \times \bS^2$. Suppose that 
	\begin{enumerate}[(i)]
	\itemsep0.5em
		\item the scalar curvature of $\g_f$ satisfies 
		\begin{equation*}
		R(\g_f) > 2 |E_f|^2_{\g_f},
		\end{equation*}
		where $E_f=Q f^{-2}\partial_s$ for some constant $Q$,
		\item $f(b)>|Q|$,
		\item the mean curvature of $\S_b$ is positive, and
		\item $\m_H^{CH}(\S_b)>|Q|$ .
	\end{enumerate}

Then, for any $m_e > \m_H^{CH}(\S_b)$, there exists a rotationally symmetric, asymptotically flat, electrically charged Riemannian manifold $(M,\gamma,E)$ with divergence-free vector field $E$ and $R(\g) \geq |E|_{\g}^2$ such that
	\begin{enumerate}[(I)]
	\itemsep0.5em
		\item its boundary $\pr M$ has a neighborhood which is isometric to $(\left[a,\frac{a+b}{2}\right], \g_f)$
		\item $(M,\gamma,E)$ is isometric to a sub-extremal Reissner--Nordstr\"om space of mass $m_e$ and charge $Q$ outside a neighborhood of the inner boundary, and
		\item if $f'>0$, $M$ can be foliated by mean convex spheres that eventually coincide with the coordinate spheres in the Reissner--Nordstr\"om space.
	\end{enumerate}
\end{prop}

\begin{proof}
	Recall that the Reissner--Nordstr\"om metric of charge $Q$ and mass $m>|Q|$ can be written as
	\begin{equation}
	\g_{m,Q}=ds^2+u_{m,Q}(s)^2g_{*},
	\end{equation} 
where $s\in[0,\infty)$ and $u_{m,Q}\colon\left[0,\infty\right)\to [m+\sqrt{m^2-Q^2},\infty)$ satisfies
	\begin{itemize}
	\itemsep0.5em
		\item $u_{m,Q}(0)=m+\sqrt{m^2-Q^2}$ and $|Q| < m$,
		\item $u'_{m,Q}(s)=\sqrt{1-\frac{2m}{u_{m,Q}(s)}+\frac{Q^2}{u_{m,Q}(s)^2}}<1$, and 
		\item $u''_m(s)=\frac{mu_{m,Q}(s)-Q^2}{u_{m,Q}(s)^3}$.
	\end{itemize}  
According to~\eqref{eq-c-hawking}, the charged Hawking mass of $\Sigma_{b}=\{b\}\times \bS^2$ in $([a,b]\times \bS^2,\g_f,E_f)$ is
\begin{equation}
m_* \definedas \m_H^{CH}(\S_b)=\frac{f(b)}{2}\left(1+\frac{Q^2}{f(b)^2}-f'(b)^2\right)
\end{equation} 

Fix $m_e>m_*$. In order to apply Lemma~\ref{lemma-gluing} to glue $([a,b]\times \bS^2,\g_{f})$ to a Reissner--Nordstr\"om manifold of mass $m_e$ and charge $Q$, we require that $u_{m_e,Q}(s_{0})>f(b)$ and $u_{m_e,Q}'(s_0) \leq f'(b)$, for some $s_0 > 0$. Both are automatically satisfied for small $s_0>0$ if $f(b)< m_e+\sqrt{m_e^2-Q^2}=u_{m_e,Q}(0)$, since $f'(b)>0=u'_{m_e,Q}(0)$, by (ii) and~\eqref{mean-collar}. Consider now the case $f(b)>m_e+\sqrt{m_e^2-Q^2}$.  Since the range of $u_{m_e,Q}$ is $[m_e+\sqrt{m_e^2-Q^2},\infty)$, for given $\epsilon>0$ there is an $s_{\epsilon}>0$ such that $u_{m_e,Q}(s_{\epsilon})=f(b)+\epsilon$. Thus we have
\begin{equation}
\begin{split}
u'_{m_e,Q}(s_{\epsilon})^2&=1-\frac{2m_e}{f(b)}+\frac{Q^2}{f(b)^2}+O(\epsilon)\\
&=1-\frac{2m_e}{f(b)}-\frac{2m_*}{f(b)}+\frac{2m_*}{f(b)}+\frac{Q^2}{f(b)^2}+O(\epsilon)\\
&=f'(b)^2-\mu+O(\epsilon)
\end{split}
\end{equation}
where $\mu \definedas 2\frac{m_e-m_\ast}{f(b)}>0$, by the definition of $m_*$. Therefore, also in this case, for  a sufficiently small $\epsilon>0$, $0\leq u'_{m_e,Q}(s_{\epsilon})<f'(b)$. 
\newpage
We now perform the deformation procedure of Lemma~\ref{bending} to obtain $\widetilde{u}_{m_e,Q}$, so that we have $R(\widetilde{u}_{m_e,Q}) > 2|\widetilde{E}_{m_e,Q}|^2_{\widetilde{u}_{m_e,Q}}$ in a small region and $\widetilde{u}_{m_e} = u_{m_e}$ outside a compact set. Using the last part of Lemma~\ref{bending}, notice that for $\delta >0$ small we have $\widetilde{u}_{m_e}(s_{\epsilon}- \delta) > f(b)$ and $\widetilde{u}_{m_e}'(s_{\veps} -\delta) < f'(b)$, where $\widetilde{u}_{m_e} = u_{m_e}(\s(s))$ ($\s$ as in Lemma~\ref{bending} with $s_0 = s_{\epsilon}$). 

Since by definition $\widetilde{u}_{m_e}(s_{\epsilon} - \delta)>|Q|$, to apply Lemma~\ref{lemma-gluing}, it remains to check that $1+\frac{Q^2}{\widetilde{u}_{m_e}(s_{\epsilon} - \delta)^{2}}-\frac{2|Q|}{\widetilde{u}_{m_e}(s_{\epsilon} - \delta)}> \widetilde{u}_{m_e}'(s_{\epsilon} - \delta)^2$. As pointed out in Remark~\ref{mCH-cond}, this condition is equivalent to $m_H^{CH}(\Sigma_{s_{\epsilon} - \delta})>|Q|$. Compute $\dot{\sigma}(s_{\veps} - \delta)^2\asdefined 1+p_{\delta}$, where $p_{\delta}>0$ and $p_{\delta}  \to 0$ as $\delta \to 0$.
\begin{align*}
	&\m_{CH}(\Sigma_{s_{\epsilon} - \delta}) \\
	&=\frac{\widetilde{u}_{m_e,Q}(s_{\epsilon} - \delta)}{2}\left(1+\frac{Q^2}{\widetilde{u}_{m_e,Q}(s_{\epsilon} - \delta)^2}-\widetilde{u}_{m_e,Q}'(s_{\epsilon} - \delta)^2\right)\\
	&=\frac{\widetilde{u}_{m_e,Q}(s_{\epsilon} - \delta)}{2}\left(1+\frac{Q^2}{\widetilde{u}_{m_e,Q}(s_{\epsilon} - \delta)^2}-\left(1-\frac{2m_e}{\widetilde{u}_{m_e,Q}(s_{\epsilon} - \delta )}+\frac{Q^2}{\widetilde{u}_{m_e,Q}(s_{\epsilon} - \delta)^2}\right)\dot{\sigma}(s_{\epsilon} - \delta)^2\right)\\
	&=\frac{\widetilde{u}_{m_e,Q}(s_{\epsilon} - \delta)}{2}\left(1+\frac{Q^2}{\widetilde{u}_{m_e,Q}(s_{\epsilon} - \delta)^2}-\left(1-\frac{2m_e}{\widetilde{u}_{m_e,Q}(s_{\epsilon} - \delta )}+\frac{Q^2}{\widetilde{u}_{m_e,Q}(s_{\epsilon} - \delta)^2}\right)(1+p_{\delta})\right)\\
	&=-\frac{\widetilde{u}_{m_e,Q}(s_{\epsilon} - \delta)}{2}p_{\delta}+m_e(1+p_{\delta})-\frac{Q^2}{2\widetilde{u}_{m_e.,Q}(s_{\epsilon} - \delta)}p_{\delta}\\
	&=m_e+\frac{p_{\delta}}{2}\left(2m_e-\widetilde{u}_{m_e,Q}(s_{\epsilon} - \delta)-\frac{Q^2}{\widetilde{u}_{m_e,Q}(s_{\epsilon} - \delta)}\right)\\
	&>m_e+\frac{p_{\delta}}{2}\left(2m_e-\widetilde{u}_{m_e,Q}(s_{\epsilon} - \delta)-|Q|\right)\\
	&>m_e+\frac{p_{\delta}}{2}\left(m_e-\widetilde{u}_{m_e,Q}(s_{\epsilon} - \delta)\right) \\
	&>m_e+\frac{p_{\delta}}{2}\left(m_e-u_{m_e,Q}(s_{\epsilon})\right). 
\end{align*} 
Therefore, for sufficiently small $\delta>0$, since $m_e>m_*>|Q|$ and $p_{\delta}$ can be made arbitrarily small, we obtain
\begin{equation}
\m_H^{CH}(\Sigma_{s_{\epsilon} - \delta})>|Q|,
\end{equation}
as desired. We now apply Lemma~\ref{lemma-gluing} with $f_1=f$ on $[a,b]$ and $f_2=\widetilde{u}_{m_e,Q}$ on $[s_{\epsilon} - \delta, s_{\epsilon} - \frac{\delta}{2}]$.
	\end{proof}

\section{Extensions with minimal boundary} \label{sec-minimal}

Given an initial data set $(M,\g,E)$, using the stability equation as in~\cite{M-S}, together with the relation between the scalar curvature of $(M,\g)$ and the electric field $E$, if $\S \cong \bS^2$ is a stable horizon, we have
\begin{equation*}
\int_{\Sigma}\left(-\Lap_{g} \vphi^2  + K(g) \vphi^2\right)dA_g \geq \int_{\S} \frac{R(\g) + |A|^2}{2} \vphi^2  dA_g \geq \int_{\Sigma}|E|^2_\gamma \vphi^2dA_g
\end{equation*}
for all $\vphi \not\equiv 0$ smooth on $\bS^2$, hence  
\begin{equation} \label{stability}
\int_{\Sigma}\left(-\Lap_{g} \vphi^2  + K(g) \vphi^2\right)dA_g  \geq \int_{\Sigma}|E|^2_\gamma \vphi^2dA_g
\end{equation}
for all $\vphi \not\equiv 0$ smooth on $\bS^2$. By setting $\vphi \equiv 1$ and applying H\"older's inequality, we obtain the area-charge relation (c.f.~\cite{Gibbons})
\begin{equation} \label{area-charge}
4\pi \geq \frac{16 \pi^2 Q^2}{ |\S|},
\end{equation}
where we have used Gauss--Bonnet and the definition of the charge contained in $\S$ given in~\eqref{total-charge-surface}.

Recall that by the variational characterization of eigenvalues of $-\Lap_g + K(g)$, its first eigenvalue is given by
\begin{equation*}
\lambda_1 \definedas \inf \left\{ \dfrac{\int_{\S}( |\grad^{g} \vphi |^2_{g} + K(g)\vphi^2)\, dA_g}{\int_{\S} \vphi^2 \, dA_g} \, \bigg\vert \, \vphi \not\equiv  0\right\}.
\end{equation*}
Note that if $c \definedas \max\limits_{M} |E|^2_{\g}$, then $\lambda_1 \geq c$ implies the stability equation~\eqref{stability}. Thus we are naturally led to consider the set
\begin{equation}
\M^{\kappa^+} \definedas \{ g \,\,\textnormal{metric on $\S \cong \bS^2$} \, | \,  \lambda_1(-\Lap_g + K(g)) > \kappa \},
\end{equation}
and we have the following lemma.

\begin{lemma}\label{path-lambda}
Let $(\S \cong \bS^2,g)$ be a 2-dimensional Riemannian manifold with metric $g$ satisfying $\lambda_1(-\Lap_g + K(g) )> \kappa$. 
Then there exists a smooth path of metrics $\{ g(t) \}_{0 \leq t \leq 1}$ connecting $g$ to a round metric, such that $\lambda_1( -\Lap_{g(t)} + K(g(t)) ) > \kappa$ for all $0 \leq t \leq 1$.
\end{lemma}

\begin{proof}
Let $r_o$ be the area radius of $(\S,g)$, that is $|\S|_{g}= 4\pi r_o^2$.  Recall that the first eigenvalue of the operator $-\Lap_g + K(g)$ can be computed via the Rayleigh quotient as
\begin{equation} \label{rayleigh-q}
\lambda_1  \definedas \inf\left\{ \frac{\int_{\S}( |\grad^{g} \vphi |^2_{g} + K(g)\vphi^2)\, dA_g}{\int \vphi^2 \, dA_{g}} \, \bigg\vert \, \vphi \not\equiv 0 \right\}.
\end{equation}
In particular, by setting $\vphi \equiv 1$, we obtain that $r_o^{-2} > \k$.

Now use the uniformization theorem to write $g= e^{2w} g_{r_o}$, where  $g_{r_o} \definedas r_o^2 g_{*}$ and $w$ is a smooth function. Let $g(t)=e^{2(1-t)w}g_{r_o}$ be a smooth path of metrics connecting $g$ to $g_{r_o}$, as in~\cite{M-S}.

Observe that by a direct computation
\begin{align*}
L(t) & \definedas \int_{\S}( |\grad^{g(t)} \vphi |^2_{g(t)} + K(g(t))\vphi^2)\, dA_{g(t)} \\
&=\int_{\S}( |\grad^{g_{r_o}} \vphi |^2_{g_{r_o}} + \(\frac{1}{r_o^2}-(1-t)\Lap_{g_{r_o}} w \)\vphi^2)\, dA_{g_{r_o}}.
\end{align*}
In particular $L(0) > \(\k + \delta\) \int_{\S} \vphi^2 \, e^{2w} dA_{g_{r_o}}$, for some  $\delta>0$, which can be picked sufficiently small so that $L(1) > (\k + \delta)  \int_{\S} \vphi^2 \, dA_{g_{r_o}}$, since $r_o^{-2} > \k > 0$. We can write $L(t)$ as 
\begin{equation*}
L(t) = tL(1) + (1-t)L(0) > t(\k + \delta)   \int_{\S} \vphi^2 \, dA_{g_{r_o}} + (1-t)\(\k + \delta\) \int_{\S} \vphi^2 \, e^{2w} dA_{g_{r_o}}.
\end{equation*}
In view of~\eqref{rayleigh-q}, to show that $\lambda_1(-\Lap_{g(t)} + K(g(t)) ) > \kappa$ it suffices to show that
\begin{equation*}
 t(\k + \delta)   \int_{\S} \vphi^2 \, dA_{g_{r_o}} + (1-t)\(\k + \delta\) \int_{\S} \vphi^2 \, e^{2w} dA_{g_{r_o}} > (\k + \delta) \int_{\S} \vphi^2 \, e^{2w(1-t)}dA_{g_{r_o}}.
\end{equation*}
This is equivalent to showing that the function
\begin{equation*}
h(t) \definedas t(\k + \delta) + (1-t)(\k + \delta)e^{2w} - (\k + \delta) e^{2w(1-t)}
\end{equation*}
is non-negative on $[0,1]$, but since $h(0) = 0 =h(1)$ and $h''(t) \leq 0$ on $[0,1]$, $h(t)$ must be positive on $(0,1)$, and the result follows.
\end{proof}

\begin{coro}\label{path-mod}
Given $(\S \cong \bS^2,g_o)$ with $\lambda_1 \definedas \lambda_1(-\Lap_{g_o}+K(g_o))>0$ and $\k>0$ such that $\lambda_1 > \k$, one can repeat the procedure of~\cite[Lemma 1.2]{M-S} to modify the path given by Lemma~\ref{path-lambda}, so that $\{ g(t) \}$ satisfies
\begin{enumerate}[(i)]
\itemsep0.5em
\item $g(0)=g_o$ and $g(1)$ is round,
\item $g'(t)=0$ for $t \in [\theta,1]$ for some $0 < \theta < 1$, and
\item $\tr_{g(t)}g'(t) = 0$ for $t \in [0,1]$.
\end{enumerate}

We note that a rescaling is necessary to achieve (i)-(iii). As a consequence, we cannot guarantee that the same lower bound $\kappa$ on $\lambda_{1}(-\Lap_{g(t)}+ K(g(t)))$ still holds along the modified path, since $\lambda_{1}(-\Lap_{g(t)}+ K(g(t)))$ along the modified path will depend on the exact values of $w$ arising in the proof of Lemma~\ref{path-lambda}. Thus, we need to define a new threshold
\begin{equation} \label{eq-lambda}
\kappa \definedas \inf_{[0,1]} \lambda_1(-\Lap_{g(t)}+ K(g(t)))
\end{equation}
which will be strictly positive, $\kappa>0$ by construction. In particular, we have that $g(t) \in \M^{\kappa^+}$ for all $t \in [0,1]$.
\end{coro}

Let $(\Sigma,g_o, H_o=0,Q_o)$ be charged Bartnik data and let $r_o$ be defined as the area radius $|\Sigma|_{g_o}=4\pi r_o^2$. Consider a path $\{g(t)\}$ as in Corollary~\ref{path-mod}. Then, as in~\cite{M-X}, we define
\begin{equation}
\alpha \definedas \frac{1}{4}\max_{[0,1]\times\Sigma}|g'(t)|^2_{g(t)}
\end{equation}
and
\begin{equation}
\beta \definedas r_o^2\min_{[0,1]\times\Sigma}K(g(t)).
\end{equation}
Observe that the area-charge relation~\eqref{area-charge} implies 
\begin{equation*}
Q_o^2 \leq r_o^2.
\end{equation*}

\newpage

The next theorem asserts the existence of asymptotically flat extensions with an electric field satisfying the dominant energy condition and such that the geometry at the boundary is prescribed and the ADM mass of the extension can be made arbitrarily close to the optimal value in the charged Riemannian Penrose inequality. We follow the main ideas from~\cite{M-S} combined with the gluing methods developed in Section~\ref{sec-gluing-tools}.

\begin{thm} \label{thm-minimal}\label{thm-main}
Let $(\Sigma\cong \bS^2,g_o, H_o=0,Q_o)$ be minimal charged Bartnik data satisfying $\lambda_1 \definedas \lambda_1\left(-\Delta_{g_o}+K(g_o)\right)> 0$,  where $\lambda_1\left(-\Delta_{g_o}+K(g_o)\right)$ denotes the first eigenvalue of the operator $-\Delta_{g_o}+K(g_o)$ on $\Sigma$, $K(g_o)$ is the Gaussian curvature of $g_o$  and let $|\Sigma|_{g_o}\asdefined 4\pi r_o^2$. Suppose that 
\begin{equation*}
Q_o^2< r_o^2
\end{equation*}
and assume furthermore that
\begin{equation*} \kappa > \frac{Q_o^2}{r_o^4},
\end{equation*}
where $\kappa$ is given by
\begin{equation*}
\kappa \definedas \inf_{[0,1]} \lambda_1(-\Lap_{g(t)}+ K(g(t)))
\end{equation*}
along a suitable path of metrics $\lbrace{g(t)\rbrace}_{0 \leq t \leq 1}$ as in Corollary~\ref{path-mod}. Then, for any 
	\begin{equation*}
	m>\m_H^{CH}(\Sigma\cong \bS^2,g_o, H_o=0,Q_o)=\sqrt{\frac{|\Sigma|_{g_o}}{16\pi}}+\sqrt{\frac{\pi}{|\Sigma|_{g_o}}}Q_o^2,
	\end{equation*}
there is an asymptotically flat, electrically charged Riemannian 3-manifold $(M,\gamma,E)$ with $R_{\gamma}\geq 2|E|_{\gamma}^2$, where $E$ is a divergence-free electric field of total charge $Q_o$, such that 
	\begin{enumerate}[(i)]
	\itemsep0.5em
		\item the boundary $\partial M$ is minimal and isometric to $(\Sigma,g_o)$,
		\item outside a compact set, $M$ coincides with the spatial Reissner-Nordstr\"om manifold with mass $m$ and charge $Q_o$, such that $m > |Q_o|$, 
		\item $M$ is foliated by mean convex spheres that eventually coincide with the coordinate spheres in the spatial Reissner-Nordstr\"om manifold, and
\item $E$ eventually coincides with the electric field of the spatial Reissner-Nordstr\"om manifold. 
	\end{enumerate}
	\end{thm}
\begin{proof}
Let $\{g(t)\}_{0 \leq t \leq 1}$ be a smooth path of metrics connecting $g_o$ to a round metric as described in Corollary~\ref{path-mod}. For each $t\in[0,1]$, let $u(t,x)>0$ be a smooth eigenfunction on $\Sigma$ corresponding to first eigenvalue $\lambda(t) \definedas \lambda_1(t)$ of $-\Delta_{g(t)}+K(g(t))$, normalized to have $L^2$-norm equal to 1. It can be checked that this choice of $u$ is smooth (see~\cite{M-S}). Let $0<\epsilon<1$ and define the collar extension
	 \begin{equation}
	 \gamma_c \definedas A^2u(t,\cdot)^2dt^2+(1+\epsilon t^2)g(t),
	 \end{equation}
where $A>0$ is a constant to be determined and $0 < \epsilon \ll 1$. By~\eqref{scalar-collar}, the scalar curvature of $\g_c$ is given by
 \begin{equation}
\begin{split}
 R(\gamma_c)&=2v(t,\cdot)^{-1}\left(-\Delta_{F(t)^2g(t)}v(t,\cdot)+K(F(t)^2g(t))v(t,\cdot)\right)\\
&\quad+v(t,\cdot)^{-2}\left(\frac{-2F'(t)^2-4F(t)F''(t)}{F(t)^2}-\frac{1}{4}|g'(t)|^2_{g(t)}+4\partial_t\log v(t,\cdot)\partial_t\log F(t)\right),
\end{split}
\end{equation}
for $F(t)=(1+\epsilon t^2)^{1/2}$ and $v(t,\cdot)=Au(t,\cdot)$. Consider the electric field $E \definedas \frac{Q_o}{f(t,\cdot)^2}\partial_t$, where $f(t,\cdot)^2 \definedas Ar_o^2u(t,\cdot)F(t)^2$, then by Lemma~\ref{lemma-e-field-collar},
	 \begin{equation}
	\text{div}_{\gamma_c}E=0,
	 \end{equation}
and the total charge contained in $\S_t$ (given by~\eqref{total-charge-surface}), is equal to $Q_o$.
	
Using the definition of $F$,$v$, and the fact 
\begin{equation}
-\Delta_{F(t)^2g(t)}(A u(t,\cdot))+\frac{1}{2}R(F(t)^2g(t))Au(t,\cdot)=F(t)^{-2}\lambda(t)A u(t,\cdot),
\end{equation} 
we have
	 \begin{align*}
	R(\gamma)&-2|E|_{\gamma}^2\\
	&=2F(t)^{-2}\lambda(t)\\
	&\quad+\frac{u(t,\cdot)^{-2}}{A^{2}}\left(\frac{-2F'(t)^2-4F(t)F''(t)}{F(t)^2}-\frac{1}{4}|g'(t)|^2_{g(t)}+4\frac{\pr_t v(t,\cdot)}{v(t,\cdot)}\frac{\pr_t F(t)}{F(t)}\right)-\frac{2Q_o^2A^2u^2(t,\cdot)}{f(t,\cdot)^4}\\
	&=\frac{2u(t,\cdot)^{-2}}{A^{2}F(t)^{2}}\bigg[A^2u(t,\cdot)^2\left(\lambda(t)-\frac{Q_o^2}{r_o^4F(t)^2}\right)-\epsilon-\frac{\epsilon}{F(t)^2}
	-\frac{1}{8}|g'(t)|^2_{g(t)}F(t)^2+2\epsilon t \frac{\pr_t u(t,\cdot)}{u(t,\cdot)}\bigg]\\
	&\geq \frac{2u(t)^{-2}}{A^{2}F(t)^{2}}\bigg[A^2\inf_{[0,1]\times\Sigma}u^2  \left(\lambda(t) -\frac{Q_o^2}{r_o^4}\right)-2-\alpha-2\sup_{[0,1]\times\Sigma}|\partial_t\log u|\bigg].
	 \end{align*}
Since $u(t,\cdot)>0$ and $\lambda(t) -\frac{Q_o^2}{r_o^4}>0$ for all $t\in[0,1]$, we have $\inf\limits_{[0,1]\times\Sigma}u^2  \left(\lambda(t) -\frac{Q^2_o}{r_o^4}\right)>0$. Now we choose $A>0$ such that
	 \begin{equation}
	 A^2\inf_{[0,1]\times\Sigma}u^2  \left(\lambda(t) -\frac{Q_o^2}{r_o^4}\right)-2-\alpha-2\sup_{[0,1]\times\Sigma}|\partial_t\log u|>0.
	 \end{equation} 
Therefore,
\begin{equation} \label{cond-g-1}
R(\gamma_c)> 2|E|_{\gamma_c}.
\end{equation}

Using~\eqref{mean-collar}, the mean curvature of $\S_t=\{ t \} \times \S$ is given by  
\begin{equation} \label{cond-g-2}
 H(t)=\frac{2F'(t)}{Au(t,\cdot) F(t)}\geq 0,
 \end{equation} 
and in particular, $H(0)=0$.
	 
By construction of the path $\{ g(t)\}_{0 \leq t \leq 1}$, we have $g(t)=r_o^2g_*$ for $t\in[\theta,1]$, which implies that $u(t,\cdot)=u(1,\cdot)\asdefined u(1)$ is a positive constant. Hence, the charged Hawking mass of $\Sigma_t$ for $\theta < t <1$ is given by
	\begin{equation}
\begin{split}
	\m_H^{CH}(\S_t)&=\sqrt{\frac{|\Sigma_t|}{16\pi}}\left(1+\frac{4\pi Q_o^2}{|\S_t|}-\frac{1}{16\pi}\int_{\S_t}H(t)^2 \, d\s_t\right)\\
	&=\frac{F(t)r_o}{2}\left(1+\frac{Q_o^2}{F(t)^2r_o^2}-\frac{r_o^2F'(t)^2}{A^2u(t,\cdot)^2}\right)\\
	&=\frac{(1+\epsilon t^2)^{1/2}r_o}{2}\left(1+\frac{Q_o^2}{(1+\epsilon t^2)r_o^2}-\frac{r_o^2\epsilon^2t^2}{A^2u(1)^2(1+\epsilon t^2)}\right).\\
\end{split}
	\end{equation}
In particular, for any  $0<\frac{\epsilon^{3/2}}{1+\epsilon}<\frac{A^2u(1)^2}{r_o^2}\left(1+\frac{Q^2_o}{r_o^2}\right)$, we have
	\begin{equation}
	\begin{split}
\m_H^{CH}(\S_1)&\leq 	\frac{(1+\sqrt{\epsilon})r_o}{2}\left(1+\frac{Q_o^2}{r_o^2}-\frac{r_o^2\epsilon^2}{A^2u(1)^2(1+\epsilon)}\right)\\
	&=\frac{r_o}{2}\left(1+\frac{Q_o^2}{r_o^2}\right)+\frac{\sqrt{\epsilon}r_o}{2}\left(1+\frac{ Q_o^2}{r_o^2}\right)-\frac{(1+\sqrt{\epsilon})r_o^3\epsilon^2}{2A^2u(1)^2(1+\epsilon)}\\
	&\leq \frac{r_o}{2}\left(1+\frac{Q_o^2}{r_o^2}\right)+\frac{\sqrt{\epsilon}r_o}{2}\left(1+\frac{ Q_o^2}{r_o^2}\right)-\frac{r_o^3\epsilon^2}{2A^2u(1)^2(1+\epsilon)}\\
	&\leq \m_H^{CH}(\S_0)+\sqrt{\epsilon} C,
	\end{split}
	\end{equation} 
where $C>0$ is a constant depending on $r_o$, $A$, $Q_o$, and $u(1)$.  Therefore for a sufficiently small $\epsilon$, 
\begin{equation}\label{m-minimal}
m>\m_H^{CH}(\Sigma,g_0, H_o=0,Q_o)+\sqrt{\epsilon} C\geq \m_H^{CH}(\Sigma_1).
\end{equation} 
	
In addition, we have 
\begin{align*}
\m_H^{CH}(\S_1) -\m_H^{CH}(\S_0)=\frac{r_o}{2}\left(\sqrt{1+\epsilon}-1+\frac{Q_o^2}{r_o^2\sqrt{1+\epsilon}}-\frac{Q_o^2}{r_o^2}-\frac{\epsilon^2 r_o^2}{A^2u(1)^2\sqrt{1+\epsilon}}\right).
\end{align*}It is straightforward to check that the right-hand side is positive by taking $\epsilon$ sufficiently small, since $Q^2<r_o^2$. Using the fact that $\m_H^{CH}(\S_0) > |Q_o|$, we obtain
\begin{equation} \label{cond-g-3}
\m_H^{CH}(\S_1) > |Q_o|.
\end{equation}

To apply Proposition~\ref{glue-collar-RN}, perform the change of variables $s(t)=Au(1)t$, then the metric for $s\in[Au(1)\theta,Au(1)]$ is
	\begin{equation} \label{collar-to-rot}
	\gamma_c=ds^2+f(s)^2g_{\ast},\qquad f(s)^2=r_o^2\left(1+\frac{\epsilon}{A^2u(1)^2}s^2\right).
	\end{equation} 
Notice that~\eqref{cond-g-1},~\eqref{cond-g-2} and~\eqref{cond-g-3} are precisely the conditions needed in Proposition~\ref{glue-collar-RN}, and by our choice $m$ (see~\eqref{m-minimal}), we can readily apply Proposition~\ref{glue-collar-RN} to construct the desired asymptotically flat, electrically charged Riemannian manifold $(M,\gamma,E)$, with mass $m$ and total charge $Q_o$, with scalar curvature satisfying $R(\gamma)>2|E|_{\gamma}$. Moreover, it has the desired boundary geometry. Since $F'(t)>0$ for $t \in (0,1)$, it is foliated by mean convex spheres which eventually will coincide with coordinate spheres of the spatial Reissner--Nordstr\"om manifold. The electric field along this manifold also eventually coincides with the electric field of the spatial Reissner--Nordstr\"om manifold. This completes the proof.
\end{proof}

\vfill

\bibliographystyle{amsplain}
\bibliography{CBH-final}
\vfill

\end{document}